\documentclass[leqno,11pt,a4paper]{article}
\usepackage[left=3.2cm, top=2.2cm,bottom=2.5cm]{geometry}
\usepackage{amsmath,amssymb,amsthm,mathrsfs,dsfont,calc,graphicx}
\usepackage[british]{babel}

\newtheorem{theorem}{Theorem}[section]

\newtheorem{proposition}[theorem]{Proposition}
\newtheorem{lemma}[theorem]{Lemma}
\newtheorem{corol}[theorem]{Corollary}
\newtheorem{remark}{Remark}

\newcommand{\V}{\operatorname{Var}}

\newcommand{\dist}{{\operatorname{dist}}}
\newcommand{\dint}{{\operatorname{d}}}

\def\R{{\mathbb{R}}}
\def\S{{\mathbb{S}}}

\def\E{{\mathbb{E}}}

\def\N{\mathbb{N}}
\def\P{{\mathbb{P}}}
\def\1{\mbox{{\rm 1 \hskip-1.4ex I}}}

\def\Y{{\mathbb{Y}}}

\def\cY{{\mathcal Y}}

\begin{document}

\title{\textbf{The scaling limit of Poisson-driven\\ order statistics with applications\\ in geometric probability}}

\author{Matthias Schulte\footnotemark[1]\, and Christoph Th\"ale\footnotemark[2]}

\date{}
\renewcommand{\thefootnote}{\fnsymbol{footnote}}
\footnotetext[1]{
Institut f\"ur Mathematik der Universit\"at Osnabr\"uck, Albrechtstra\ss e 28a, 49076 Osnabr\"uck, Germany. E-mail: matthias.schulte@uni-osnabrueck.de}

\footnotetext[2]{
Institut f\"ur Mathematik der Universit\"at Osnabr\"uck, Albrechtstra\ss e 28a, 49076 Osnabr\"uck, Germany. E-mail: christoph.thaele@uni-osnabrueck.de}

\maketitle

\begin{abstract}
Let $\eta_t$ be a Poisson point process of intensity $t\geq 1$ on some state space $\Y$ and $f$ be a non-negative symmetric function on $\Y^k$ for some $k\geq 1$. Applying $f$ to all $k$-tuples of distinct points of $\eta_t$ generates a point process $\xi_t$ on the positive real-half axis. The scaling limit of $\xi_t$ as $t$ tends to infinity is shown to be a Poisson point process with explicitly known intensity measure. From this, a limit theorem for the the $m$-th smallest point of $\xi_t$ is concluded. This is strengthened by providing a rate of convergence. The technical background includes Wiener-It\^o chaos decompositions and the Malliavin calculus of variations on the Poisson space as well as the Chen-Stein method for Poisson approximation. The general result is accompanied by a number of examples from  geometric probability and stochastic geometry, such as Poisson $k$-flats, Poisson random polytopes, random geometric graphs and random simplices. They are obtained by combining the general limit theorem with tools from convex and integral geometry.
\end{abstract}
\begin{flushleft}%\footnotesize
\end{flushleft}
% Statistics of extreme values, tail inference 62G32
% Geometric Probability and Stochastic Geometry 60D05
% Functional limit theorems 60F17

% Point processes 60G55
% Stochastic calculus of variations and Malliavin calculus 60H07

%----------------------SECTION------------------------

\section{Introduction and main result}

In the paper \cite{GJ} by Grimmett and Janson, the authors consider the areas of all triangles formed by a fixed number of i.i.d. random points in a (convex) planar domain. They show that after re-scaling the distribution of the smallest triangle area converges to an exponential distribution and, moreover, that the entire collection of all triangle areas converges to a homogeneous Poisson point process on the positive real-half axis as the number of points gets large.

The purpose of the current paper is to establish a framework, which allows to deal with considerably more general situations and can be applied to a broad class of examples, including higher-dimensional versions of the main result from \cite{GJ} mentioned above. We also replace the fixed number of random points by a Poisson point process (with a possibly infinite number of points), making thereby available the powerful Wiener-It\^o chaos decomposition and the Malliavin calculus of variations for Poisson functionals.

We are now going to discuss our main result and its framework in detail. To this end, fix some Borel measurable space $(\Y,\cY)$ with a non-atomic $\sigma$-finite measure $\lambda$. By $\eta_t$ we denote a Poisson point process on $\Y$ with intensity measure $\lambda_t=t\lambda$ and $\eta_{t,\neq}^{k}$, $k\geq 1$, stands for the set of all $k$-tuples of distinct points of $\eta_t$. (As usual in point process theory, a point process is a random measure, which is -- by abuse of notation -- identified with its support, so that $y\in\eta_t$ means that $y\in\Y$ is charged by the random measure $\eta_t$.) Let further $f:\Y^k\rightarrow\R$ be a non-negative measurable function that is invariant under permutations of the arguments and satisfies 
\begin{equation}\label{eq:conditionf}
\lambda^k(f^{-1}([0,x]))<\infty \text{ for all } x>0.
\end{equation}
The Poisson point process $\eta_t$ and the function $f$ induce a collection of points
$$ \xi_t=\{f(y_1,\ldots,y_k):(y_1,\ldots,y_k)\in\eta_{t,\neq}^k\}$$
on the positive real half-axis $\R_+$. Because of the symmetry of $f$, every $f(y_1,\hdots,y_k)$ also occurs for permutations of the argument $(y_1,\hdots,y_k)$. However, we count the point $f(y_1,\hdots,y_k)$ for every subset $\{y_1,\hdots,y_k\}\subset \eta_t$ only once. The collection $\xi_t$ might still have multiple points if there are several subsets having the same value under $f$. By \eqref{eq:conditionf}, $\xi_t$ is locally finite (and vice versa), whence $\xi_t$ is a point process on the half-line $\R_+$. 

 We order the points of $\xi_t$ from the left to the right with respect to the natural ordering on $\R_+$ and denote by $F^{(m)}_t$ the distance of the $m$-th point of $\xi_t$ to the origin, i.e. $F_t^{(m)}$ is the $m$-th order statistic of $f$ applied to $\eta_{t,\neq}^k$. (We put $F^{(m)}_t=+\infty$, if $\xi_t$ has less than $m$ points.) In the poissonized version of the case considered in \cite{GJ} and described at the beginning, $\Y=\R^2$, $k=3$ and $f(y_1,y_2,y_3)$ is the area of the triangle with vertices $y_1$, $y_2$ and $y_3$.

For $\gamma>0$, $t\geq 1$ and $x>0$, we denote by $\alpha_t(x)$ the mean number of $k$-tuples $(y_1,\hdots,y_k)$ of $\eta_{t,\neq}^k$ for which $f(y_1,\hdots,y_k)\leq x t^{-\gamma}$, i.e.
\begin{equation}\label{eq:alphatdef}
\begin{split}
\alpha_t(x) &= {1\over k!}\; \E \sum_{(y_1,\hdots,y_k)\in\eta_{t,\neq}^k} \1(f(y_1,\hdots,y_k)\leq x t^{-\gamma})\\
&=\frac{1}{k!}\int\limits_{\Y^k} \1(f(y_1,\hdots,y_k)\leq x t^{-\gamma})\;\lambda_t^k(\dint(y_1,\hdots,y_k)),
\end{split}
\end{equation}
where Campbell's Theorem for point processes is used to obtain the equality. Let us also introduce
\begin{equation}\label{eq:rtdef}
r_t(x)=\sup_{\substack{y_1,\hdots,y_{k-j}\in\Y\\ 1\leq j\leq k-1}}\lambda_t^j(\{(\hat{y}_1,\hdots,\hat{y}_j)\in\Y^j: f(\hat{y}_1,\hdots, \hat{y}_j,y_1,\hdots,y_{k-j})\leq x t^{-\gamma}\}),
\end{equation}
which plays a crucial r\^ole in the locality condition below. We are now prepared to present our main result.

\begin{theorem}\label{thm:general}
Let $\alpha_t(x)$ and $r_t(x)$ as in \eqref{eq:alphatdef} and \eqref{eq:rtdef} for some $\gamma>0$. Assume that there are constants $\beta,\tau>0$ such that
\begin{equation}\label{eq:alpha}
\lim\limits_{t\rightarrow\infty}\alpha_t(x)=\beta x^{\tau},\qquad x>0
\end{equation}
and that the locality condition
\begin{equation}\label{eq:condition}
\lim\limits_{t\rightarrow\infty}r_t(x)=0,\qquad x>0
\end{equation}
holds.
\begin{itemize}
 \item [a)] The scaling limit as $t\rightarrow\infty$ of the point processes $t^\gamma \xi_t$ is a Poisson point process $\xi$ on $\R_+$ with intensity measure
$$\nu(B)=\beta\tau\int\limits_B u^{\tau-1}\;\dint u,\qquad B\subset \R_+\;{\rm Borel},$$ i.e. $t^\gamma\xi_t$ converges as $t\rightarrow\infty$ in distribution to $\xi$.
\item [b)] For every $x>0$ there is a constant $C_{f,x}>0$ depending on $f$ and $x$ such that all order statistics $F_t^{(m)}$, $m\geq 1$, satisfy
$$\left|\P(t^\gamma F_t^{(m)}>x)-e^{-\beta x^{\tau}}\sum_{i=0}^{m-1}\frac{(\beta x^{\tau})^i}{i!}\right|\leq |\beta x^\tau-\alpha_t(x)|+C_{f,x}\sqrt{r_t(x)}$$
for all $t\geq 1$.
\end{itemize}
\end{theorem}

The Poisson point process $\xi$ in Theorem \ref{thm:general} a) has the power law intensity function $\beta\tau u^{\tau-1}$ and is also known as \textit{Weibull process} in the literature, because the distance from the origin to the first point of $\xi$ follows a Weibull distribution with survival function $e^{-\beta x^\tau}$; cf. \cite{Miller}. This can be rephrased by saying that the re-scaled minimum functional $t^\gamma F_t^{(1)}$ with 
\begin{equation}\label{eq:F1}
F_t^{(1)}=\min_{(y_1,\hdots,y_k)\in\eta^k_{t,\neq}}f(y_1,\hdots ,y_k)
\end{equation}
being the first order statistic is asymptotically Weibull distributed as $t\rightarrow\infty$. This is of special interest for many applications as considered below. We remark that the point process $\xi$ is homogeneous (this means that its intensity measure is a constant multiple of the standard Lebesgue measure on $\R_+$) with intensity $\beta>0$ if and only if $\tau=1$, in which case the mentioned Weibull distribution is nothing than an exponential distribution with parameter $\beta$.

The Wiener-It\^o chaos decomposition as outlined by Last and Penrose in \cite{LastPenrose2011} has stimulated a number of applications in geometric probability and stochastic geometry, that were concerned with \textit{central} limit theorems; cf. \cite{DFR11,LRPeccati,RS11,S11}. The current paper turns to point process convergence, \textit{non-central} limit theorems and extreme values and continues the works \cite{LRPeccati,RS11} by Lachieze-Rey, Peccati, Reitzner and Schulte, where so-called Poisson U-statistics have been investigated. The Wiener-It\^o chaos decomposition of Poisson U-statistics will also be in the background of the results obtained here, since we investigate an associated auxiliary Poisson U-statistic rather than the original problem. For this auxiliary functional we first prove a Poisson limit theorem, Proposition \ref{prop:PoissonApprox} below, with a rate measured by the total variation distance. The main tool for deriving this result is the remarkable paper \cite{Peccati11} by Peccati, who combined the Malliavin calculus of variations for Poisson functionals with the Chen-Stein method for Poisson approximation. Background material for these techniques are the paper \cite{NualartVives90} by Nualart and Vives and the monograph \cite{BarbourHolstJanson92} by Barbour, Holst and Janson. The Poisson convergence in turn implies our non-central limit theorem, Theorem \ref{thm:general} b) above. The full scaling limit is derived by general point process theory as described in Chapter 16 of Kallenberg's book \cite{Kallenberg}. Other references dealing with the Poisson point process approximation are the papers \cite{Barbour88,BarbourBrown} by Barbour and Brown, Janson's classical work \cite{Janson} and once more \cite{BarbourHolstJanson92}. 

In our examples presented in Section \ref{secEXAMPLES}, we will apply Theorem \ref{thm:general} to problems having a geometric flavor. These are:
\begin{itemize}
 \item[1.] \textit{Non-intersecting Poisson $k$-flats in $\R^d$ ($k<d/2$):} Here, we consider a compact convex set $W\subset\R^d$ and the distances between all pairs of distinct $k$-flats hitting $W$;
 \item[2.] \textit{Intersecting Poisson $k$-flats in $\R^d$ ($k\geq d/2$):} Given the $d$-dimensional unit ball $B^d$, we investigate the $j$-th intrinsic volumes of the intersection of $B^d$ with the lower-dimensional flats of the intersection process of the Poisson $k$-flats;
 \item[3.] \textit{Poisson polytope on the unit sphere:} The random polytope, which is given by the convex hull of a Poisson point process on the $(d-1)$-dimensional unit sphere is considered, in particular the length of its shortest edge;
 \item[4.] \textit{Random geometric graphs:} Here, a Poisson point process in a compact convex set $W\subset\R^d$ is given and the edge lengths of a family of random geometric graphs constructed out of these points is investigated;
 \item[5.] \textit{Random simplices I:} We consider the volumes of all $d$-dimensional random simplices that can be formed by all $(d+1)$-tuples of distinct points of a Poisson point process in a compact convex set $W\subset\R^d$;
 \item[6.] \textit{Random simplices II:} Given a Poisson process of hyperplanes hitting a compact convex set $W\subset\R^d$, we deal with the volumes of all $d$-dimensional random simplices that can be formed by any $d+1$ of these hyperplanes.
\end{itemize}
In order to apply Theorem \ref{thm:general} to these concrete situations, we will have to identify $\gamma$, $\beta$ and $\tau$ that appear in the limit of \eqref{eq:alphatdef} and to check the locality condition \eqref{eq:condition}. Typically, to check the locality condition is more or less a routine task, whereas for determining the normalizing constants one often needs more delicate arguments adapted to the particular examples. The exact computations rely in our cases on the classical Crofton formula, Steiner's formula and its relatives from convex geometry and on an integral-geometric transformation of Blaschke-Petkantschin type and we refer the reader to \cite{SW} for these geometric tools.

Let us finally mention some other closely related work. One of the classical references for extreme values and the point processes connection is Resnick's monograph \cite{Resnick}. A Poisson process limit theorem for the order statistics of i.i.d. random variables is the content of the paper \cite{Miller} of Miller, where non-homogeneous Poisson point processes on the real half-axis similar to those in our Theorem \ref{thm:general} show up in the limit. Lao and Mayer have studied in \cite{LM07} so-called U-max statistics. They correspond to our functional (\ref{eq:F1}) with the minimum replaced by the maximum. Moreover, the binomial point process has been used instead of the Poisson point process considered here. The diameter of such a random sample is the content of the paper \cite{MM07} by Mayer and Molchanov; see also the related work \cite{Henze82,Henze96} by Henze and Klein and again \cite{LM07}. The minimal distance of points in a binomial point process has been investigated in the classical paper \cite{SilvermanBrown} by Silverman and Brown, where a result similar to our Theorem \ref{thm:general} b) without rate of convergence has been obtained; see also \cite{KMT}.

The paper is structured as follows: In the next section, we present our Examples 1--6 mentioned above in full detail. In Section \ref{secCHAOS} we recall some basic facts about chaos decompositions, the Malliavin calculus of variations and a result for the Poisson approximation on the Poisson space that is needed in our further arguments. The proof of our general result, Theorem \ref{thm:general}, is the content of the final Section \ref{secPOISSONAPPROXIMATION}.

\section{Applications in geometric probability}\label{secEXAMPLES}

Let us fix some general notation before turning to the examples. For a (full dimensional) set $W\subset{\R}^d$, $d\geq 1$, we denote by $[W]_k$ the set of all $k$-dimensional affine subspaces of $\R^d$ that have non-empty intersection with $W$. In integrals of the type $\int_{[W]_k}f(E)\; dE$ with a real-valued measurable function $f$ on $[W]_k$, $dE$ stands for integration with respect to the Haar measure on $[W]_k$ with a normalization as in \cite{SW}. If in addition $W$ is convex, we denote by $V_j(W)$, $0\leq j\leq d$, the intrinsic volume of order $j$ in the sense of classical convex geometry. In particular, $V_d(W)$ is the volume of $W$, $2V_{d-1}(W)$ its surface area, $V_1(W)$ a constant multiple of its mean width and $V_0(W)=1$. Further, $B^d_r(y)$ stands for the $d$-dimensional ball with radius $r>0$ and center $y\in\R^d$, $B^d_r=B^d_r(0)$ and $B^d=B_1^d$. Moreover, $\kappa_d=V_d(B^d)$. The Euclidean distance between two points $y_1,y_2\in\R^d$ is $\dist(y_1,y_2)$ and for $W$ as above we put $\dist(y,W)=\inf\{\dist(y,w):w\in W\}$ for $y\in\R^d$. For $r>0$ let us define the outer and inner parallel sets $W_r=\{y\in\R^d: \dist(y,W)\leq r\}$ and $W_r^-=\{y\in W: \dist(y,\partial W)\geq r\}$, where $\partial W$ denotes the boundary of $W$.

\paragraph{1. Non-intersecting Poisson $k$-flats.} Poisson point processes on the space of $k$-dimensional affine subspaces of $\R^d$ are a classical topic studied in stochastic geometry; see \cite{BaumstarkLast,HugLastWeil} and also \cite{SW} and the references cited therein. To measure the `closeness' or the `denseness' the so-called proximity has been introduced in \cite{Schneider99} for the case $k<d/2$, where the flats do not intersect each other with probability one. We propose to measure such a quantity by the minimal distance of the flats hitting a convex test set or, more generally, by the order statistics induced by all distances between two distinct flats. As we assume stationarity of the Poisson flats, our result does not depend on the position of the test set in space. Moreover, we also assume isotropy in order to make available tools from integral geometry.

More formally, let $\eta_t$ be a stationary and isotropic process of Poisson $k$-flats in $\R^d$ of intensity $t\geq 1$ such that $k<d/2$ and $d\geq 1$. Fix a compact and convex test set $W\subset\R^d$ with volume $V_d(W)>0$ and define the distance between two $k$-flats $E$ and $F$ hitting $W$ as $$\dist_W(E,F)=\min_{y_1\in E\cap W,\;y_2\in F\cap W}\dist(y_1,y_2).$$ 
Note that $\dist_W(E,F)$ measures the distance of $E$ and $F$ within $W$ and not the usual distance between the flats. We find this approach more natural as it can happen that two flats have a large distance in $W$, but become close to each other far away from the test set. The Poisson $k$-flat process $\eta_t$ and $\dist_W(\cdot,\cdot)$ generate a point process 
$$\xi_t=\{\dist_W(E,F):(E,F)\in\eta_{t,\neq}^2\;{\rm with}\; E,F\in[W]_k\}$$ on the real half axis. By $D^{(m)}_t$, we denote the distance of the $m$-th point of $\xi_t$ to the origin, which is the $m$-th smallest distance between two $k$-flats hitting $W$.

\begin{theorem}\label{thm:proximity} 
Define $$\beta=\frac{{d-k \choose k}\kappa_{d-k}^2}{2{d \choose k} \kappa_d}V_{d}(W).$$
\begin{itemize}
\item [a)] The re-scaled point processes $t^{2/(d-2k)}\xi_t$ converge as $t\rightarrow\infty$ in distribution to a Poisson point process whose intensity measure is $$B\mapsto\beta (d-2k)\int\limits_B u^{d-2k-1}\; \dint u,\qquad B\subset \R_+\;{\rm Borel}.$$
\item [b)] For every $x>0$ there is a constant $C>0$ depending on $x$, $d$, $k$ and $W$ such that
\begin{eqnarray*}
\left|\P(t^{2/(d-2k)}D_t^{(m)}>x)-e^{-\beta x^{d-2k}}\sum_{i=0}^{m-1}\frac{(\beta x^{d-2k})^i}{i!} \right| \leq C t^{-\min\{2/(d-2k),1/2\}}
\end{eqnarray*}
for $t\geq 1$ and $m\geq 1$. In particular, $t^{2/(d-2k)}D_t^{(1)}$ converges as $t\rightarrow\infty$ in distribution to a Weibull distributed random variable with survival function $e^{-\beta x^{d-2k}}$.
\end{itemize}
\end{theorem}

\begin{proof}
In order to apply Theorem \ref{thm:general}, we need to compute the limit of
\begin{equation}\label{expectationproximity}
\alpha_t(x)=\frac{t^2}{2}\int\limits_{[W]_k}\int\limits_{[W]_k}\1(\dist_W(E,F)\leq xt^{-\gamma})\;\dint F \; \dint E
\end{equation}
as $t\rightarrow\infty$. Applying Crofton's formula \cite[Thm. 5.1.1]{SW}, we obtain
\begin{equation}\label{eq:croftonproximity}
\begin{split}
\int\limits_{[W]_k}\1(\dist_W(E,F)\leq xt^{-\gamma})\;\dint F & =  \int\limits_{[W]_k}\1(((E\cap W)_{xt^{-\gamma}}\cap W)\cap F\neq\emptyset)\;\dint F\\
&= {\kappa_k\kappa_{d-k}\over{d\choose k}\kappa_d}V_{d-k}((E\cap W)_{xt^{-\gamma}}\cap W).
\end{split}
\end{equation}
Observe now that
\begin{equation}\label{eq:boundExp1}
V_{d-k}((E\cap W^-_{xt^{-\gamma}})_{xt^{-\gamma}})\leq V_{d-k}((E\cap W)_{xt^{-\gamma}}\cap W)\leq V_{d-k}((E\cap W)_{xt^{-\gamma}})
\end{equation}
and that a version of Steiner's formula \cite[Thm. 14.2.4]{SW} leads to
\begin{eqnarray}\label{eq:Steiner1}
 V_{d-k}((E\cap W^-_{xt^{-\gamma}})_{xt^{-\gamma}}) & = & \sum_{j=0}^k \frac{\kappa_{d-j}}{\kappa_k}{d-j \choose k}V_j(E\cap W^-_{xt^{-\gamma}}) (xt^{-\gamma})^{d-k-j},\\ \label{eq:Steiner2}
 V_{d-k}((E\cap W)_{xt^{-\gamma}}) & = & \sum_{j=0}^k \frac{\kappa_{d-j}}{\kappa_k}{d-j \choose k}V_j(E\cap W) (xt^{-\gamma})^{d-k-j}.
\end{eqnarray}
Combining \eqref{expectationproximity} with \eqref{eq:boundExp1}, \eqref{eq:Steiner1} and \eqref{eq:Steiner2} and using once more Crofton's formula yield
$$\sum_{j=0} ^k \frac{{k \choose j} {d-j \choose k}}{2{d \choose k}{d \choose k-j}}\frac{\kappa_k \kappa_{d-k}\kappa_{d-j}\kappa_{d+j-k}}{\kappa_j\kappa_d^2} V_{d-k+j}(W^-_{xt^{-\gamma}})t^2(xt^{-\gamma})^{d-k-j}$$
$$\leq \alpha_t(x)\leq \sum_{j=0} ^k \frac{{k \choose j} {d-j \choose k}}{2{d \choose k}{d \choose k-j}}\frac{\kappa_k \kappa_{d-k}\kappa_{d-j}\kappa_{d+j-k}}{\kappa_j\kappa_d^2} V_{d-k+j}(W)t^2(xt^{-\gamma})^{d-k-j}.$$
This together with the monotonicity of the intrinsic volumes leads to the inequality
\begin{eqnarray*}
&&\left|\alpha_t(x)- \frac{{d-k \choose k}\kappa_{d-k}^2}{2{d \choose k}\kappa_d}V_{d}(W)t^2(xt^{-\gamma})^{d-2k}\right|\\ & \leq & \frac{{d-k \choose k}\kappa_{d-k}^2}{2{d \choose k}\kappa_d} \left(V_{d}(W)-V_{d}(W^-_{xt^{-\gamma}})\right)\\ && + \sum_{j=0}^{k-1} \frac{{k \choose j} {d-j \choose k}}{2{d \choose k}{d \choose k-j}}\frac{\kappa_k \kappa_{d-k}\kappa_{d-j}\kappa_{d+j-k}}{\kappa_j\kappa_d^2} V_{d-k+j}(W)t^2(xt^{-\gamma})^{d-k-j}.
\end{eqnarray*}
Again, by Steiner's formula it follows that 
$$|V_{d}(W)-V_{d}(W^-_{xt^{-\gamma}})|\leq \sum_{j=0}^{d-1}\kappa_{d-j}V_j(W)(xt^{-\gamma})^{d-j}.$$
Choosing $\gamma=2/(d-2k)$, we see that
\begin{equation}\label{eq:inequalityExp1}
\lim_{t\rightarrow\infty}\alpha_t(x)=\beta x^{d-2k} \quad \text{and} \quad |\alpha_t(x)-\beta x^{d-2k}|\leq c_1(k,d,W) (x+x^d)t^{-2/(d-2k)}
\end{equation}
with a suitable constant $c_1(k,d,W)>0$ depending on $k$, $d$ and $W$.
Using (\ref{eq:croftonproximity}) and the monotonicity of $V_{d-k}$, we find that
\begin{eqnarray*}
r_t(x) & = & \sup_{E\in[W]_k}t\int\limits_{[W]_k}\1(\dist_W(E,F)\leq xt^{-\gamma})\;\dint F\\
&=& \sup_{E\in[W]_k}t{\kappa_k\kappa_{d-k}\over{d\choose k}\kappa_d}V_{d-k}((E\cap W)_{xt^{-\gamma}}\cap W)\\
& \leq & \sup_{E\in[W]_k}t{\kappa_k\kappa_{d-k}\over{d\choose k}\kappa_d}V_{d-k}((E\cap W)_{xt^{-\gamma}}).
\end{eqnarray*}
Combining this once more with \cite[Thm. 14.2.4]{SW}, we obtain
\begin{equation}\label{eq:rtBSP1} \begin{split}
r_t(x) & \leq  t{\kappa_k\kappa_{d-k}\over{d\choose k}\kappa_d}\sum_{j=0}^{k}(xt^{-\gamma})^{d-k-j}{d-j \choose k}\frac{\kappa_{d-j}}{\kappa_k}\sup_{E\in[W]_k}V_j(E\cap W)\\ 
&\leq  c_2(k,d,W)t((xt^{-\gamma})^{d-k}+(xt^{-\gamma})^{d-2k}) \end{split}
\end{equation}
with a suitable constant $c_2(k,d,W)>0$. Substituting $\gamma=2/(d-2k)$ in \eqref{eq:rtBSP1} leads to
\begin{equation}\label{eq:helpxxx}
r_t(x)\leq c_2(k,d,W)(x^{d-k}+x^{d-2k})t^{-1} 
\end{equation}
for $t\geq 1$. Now Theorem \ref{thm:general} can be applied, which completes the proof.
\end{proof}
\begin{remark}\rm
A glance at \eqref{eq:inequalityExp1} and \eqref{eq:helpxxx} shows that the constant $C$ in Theorem \ref{thm:proximity} b) can be chosen in such a way that $C=\hat{C}(x+x^d)$ with $\hat{C}$ being independent of $x$. However, $x+x^d$ is not uniformly bounded so that we cannot take the supremum over all $x>0$ as would be of interest to get a rate of convergence measured by the usual Kolmogorov distance. A similar comment also applies to Examples 2--4 below.
\end{remark}

\paragraph{2. Intersecting Poisson $k$-flats.} As in the previous example we investigate the restriction of a stationary and isotropic Poisson $k$-flat process $\eta_t$ in $\R^d$ ($d\geq 2$) of intensity $t\geq 1$. But this time we focus on the case $k\geq d/2$ and restrict to the $d$-dimensional unit ball $B^d$ in place of a general convex observation window. By virtue of the parameter choice, these flats intersect with probability one and we can define the intersection process of $\eta_t$ of order $\ell$, where $\ell$ is such that $\ell(d-k)\leq d$. This is obtained by taking the intersection of any $\ell$-tuple of distinct $k$-flats of $\eta_t$; cf. \cite{SW}. By $V_j$ we denote again the intrinsic volume of degree $j$ and put 
$$\xi_t=\{V_j(E_1\cap\ldots\cap E_\ell\cap B^d): (E_1,\ldots,E_\ell)\in\eta_{t,\neq}^\ell \text{ and }E_1\cap\hdots\cap E_\ell\cap B^d\neq \emptyset\}$$
for $j\in\{1,\ldots,d-\ell(d-k)\}$. Let further $V_{j,t}^{(m)}$ be the distance of the $m$-th smallest element of $\xi_t $ to the origin.
\begin{theorem}
Define $$\beta={\kappa_d(\ell(d-k)-1)!\over 2\kappa_{d-\ell(d-k)}\ell!}{d\choose
\ell(d-k)}{d-\ell(d-k)\choose j}^{-{1\over j}}\left({k!\kappa_k\over d!\kappa_d}\right)^\ell
\left({\kappa_{d-\ell(d-k)-j}\over\kappa_{d-\ell(d-k)}}\right)^{1\over
j}.$$
\begin{itemize}
 \item[a)] The point processes $t^{j\ell/2}\xi_t $ converge as $t\rightarrow\infty$ in distribution to a Poisson point process whose intensity measure is $$B\mapsto 2\beta j^{-1}\int\limits_Bu^{(2-j)/j}\;\dint u,\qquad B\subset \R_+\;{\rm Borel}.$$
 \item[b)]  For every $x>0$ there is a constant $C>0$ depending on $x$, $d$, $k$, $\ell$ and $j$ such that
$$\left|\P(t^{j\ell/2}V_{j,t}^{(m)}>x)-e^{-\beta x^{2/j}}\sum_{i=0}^{m-1} \frac{(\beta x^{2/j})^i}{i!} \right|\leq C t^{-1/2}$$
for $t\geq 1$ and $m\geq 1$. In particular, $t^{j\ell/2}V_{j,t}^{(1)}$ converges to a random variable that is Weibull distributed with survival function $e^{-\beta x^{2/j}}$.
\end{itemize}
\end{theorem}
\begin{proof}
In this example, $\alpha_t(x)$ is given by 
\begin{equation}\label{eqn:kflatsintersecting}
\alpha_t(x)={t^\ell\over \ell!}\int\limits_{[B^d]_k}\ldots\int\limits_{[B^d]_k}\1(0<V_j(E_1\cap\ldots\cap E_\ell\cap B^d)\leq xt^{-\gamma})\; \dint E_\ell\ldots \dint E_1.
\end{equation}
The intersection $E_1\cap\ldots\cap E_\ell\cap B^d$ is a $d-\ell(d-k)$-dimensional ball. Its $j$-th intrinsic volume is less or equal than $xt^{-\gamma}$ if its radius is less than $\varrho:=(xt^{-\gamma}/V_j(B^{d-\ell(d-k)}))^{1/j}$. This happens if and only if the distance of the $d-\ell(d-k)$-dimensional ball to the origin is greater than $\sqrt{1-\varrho^2}$. Thus, for fixed $E_1,\ldots,E_{\ell-1}$ the innermost integral in (\ref{eqn:kflatsintersecting}) can be written as
$$\int\limits_{[B^d]_k}\1(E_1\cap\ldots\cap E_{\ell-1}\cap E_\ell\cap B^d\neq\emptyset)-\1(E_1\cap\ldots\cap E_{\ell-1}\cap E_\ell\cap B_{\sqrt{1-\varrho^2}}^d\neq\emptyset)\; \dint E_\ell$$
$$=\varsigma_{d,0,k}(V_{d-k}({E_1\cap\ldots\cap E_{\ell-1}\cap B^d})-V_{d-k}({E_1\cap\ldots\cap E_{\ell-1}\cap B_{\sqrt{1-\varrho^2}}^d})),$$
where, more generally,
$$\varsigma_{d,i,k} = {k!(d-k+i)!\kappa_k\kappa_{d-k+i}\over d!i!\kappa_d\kappa_i}.$$
Applying now $(\ell-1)$-times the Crofton formula \cite[Thm. 5.1.1]{SW}, we obtain 
$$\alpha_t(x)={t^\ell\over \ell!}\left(\prod_{i=0}^{\ell-1}\varsigma_{d,i(d-k),k}\right)(V_{\ell(d-k)}(B^d)-V_{\ell(d-k)}(B^d_{\sqrt{1-\varrho^2}})).$$
The homogeneity of the intrinsic volumes implies that
$$V_{\ell(d-k)}(B^d)-V_{\ell(d-k)}(B^d_{\sqrt{1-\varrho^2}})=V_{\ell(d-k)}(B^d)(1-(1-\varrho^2)^{\ell(d-k)/2}).$$
Using the asymptotic expansion
\begin{equation}\label{eq:Ex2asymptotic}
1-(1-\varrho^2)^{\ell(d-k)/2}=\frac{1}{2}\ell(d-k)\varrho^2-\frac{1}{8}\ell(d-k)(\ell(d-k)-2)\varrho^4+O(\varrho^5),
\end{equation}
for $\varrho\rightarrow 0$, we find that the latter behaves like $\frac{1}{2}\ell(d-k)V_{\ell(d-k)}(B^d)\varrho^2$. Taking $\gamma=j\ell/2$ and substituting the expression for $\varrho$ leads to
$$\alpha(x)=\lim_{t\rightarrow\infty}\alpha_t(x)={\ell(d-k)\over 2\ell!}\left(\prod_{i=0}^{\ell-1}\varsigma_{d,i(d-k),k}\right)V_{\ell(d-k)}(B^d)V_j(B^{d-\ell(d-k)})^{-2/j}x^{2/j},$$
which equals $\beta x^{2/j}$ as some elementary computation shows. Moreover, using once more \eqref{eq:Ex2asymptotic} and substituting $\varrho$, we obtain
$$ |\alpha-\alpha_t(x)| \leq c_1 t^\ell \varrho^4 = c_1(x) t^{-\ell}$$
for $t\geq 1$ and some constants $c_1, c_1(x)>0$. A similar calculation shows that
\begin{eqnarray*}
r_t(x) & = & \sup_{\substack{E_1,\ldots,E_i\in[B^d]_k\\ 1\leq i\leq\ell-1}}t^{\ell-i}\int\limits_{[B^d]_k^{\ell-i}}\1(0<V_j(E_1\cap\ldots\cap E_i\cap F_1\cap\ldots\\ && \hspace{5cm}\ldots\cap F_{\ell-i}\cap B^d)\leq xt^{-\gamma})\; \dint F_1\ldots \dint F_{\ell-i}\\
&\leq & \max_{1\leq i\leq \ell-1} c^{(i)}(x)t^{-i} \leq c_2(x) t^{-1}
\end{eqnarray*}
for $t\geq 1$ with suitable constants $c^{(1)}(x),\hdots,c^{(\ell-1)}(x),c_2(x)>0$. The assertion is now a consequence of Theorem \ref{thm:general}.
\end{proof}

\paragraph{3. Poisson polytope on the unit sphere.} Let us consider an isotropic Poisson point process $\eta_t$ on the $(d-1)$-dimensional unit sphere $\S^{d-1}$ ($d\geq 2$) having intensity $t\geq 1$. The convex hull of all points of $\eta_t$ is the \textit{Poisson polytope} with vertices on the sphere. The convex hull of a random point set is one of the most intensively studied models in geometric probability; see Chapter 8 in \cite{SW} and the references cited therein. Random polytopes with a fixed number of vertices on the boundary of a convex body were investigated in \cite{R02}, whereas \cite{BR10} deals with the general Poisson polytope.

We denote by $L_t^{(m)}$ the distance of the $m$-th smallest element of the point process
$$\xi_t=\{\dist(y_1,y_2):(y_1,y_2)\in\eta_{t,\neq}^2\}$$
to the origin, in particular $$L_t^{(1)}=\min_{(y_1,y_2)\in\eta_{t,\neq}^2}\dist(y_1,y_2).$$ The geometry of $\S^{d-1}$ ensures that $L_t^{(1)}$ is the length of the shortest edge of the Poisson polytope. On the other hand, $L_t^{(m)}, m\geq 2,$ is not necessarily the length of the $m$-th shortest edge of the Poisson polytope since the related line can lie within the interior of the Poisson polytope and does not need to be an edge of it.
\begin{theorem}
Put $\beta={d\over 2}\kappa_d\kappa_{d-1}$. 
\begin{itemize}
 \item[a)] The point processes $t^{2/(d-1)}\xi_t $ converge in distribution as $t\rightarrow\infty$ to a Poisson point process on the positive real half-axis whose intensity measure is given by $$B\mapsto\beta(d-1)\int\limits_Bu^{d-2}\;\dint u,\qquad B\subset \R_+\;{\rm Borel}.$$
 \item[b)] For every $x>0$ there is a constant $C>0$ depending on $x$ and $d$ such that
$$\left|\P(t^{2/(d-1)}L_t^{(m)}>x)-e^{-\beta x^{d-1}}\sum_{i=0}^{m-1}\frac{(\beta x^{d-1})^i}{i!}\right|\leq C t^{-\min\{2/(d-1),1/2\}}$$
for $t\geq 1$ and $m\geq 1$. In particular, $t^{2/(d-1)}$ times the length of the shortest edge of the Poisson polytope follows asymptotically a Weibull distribution with survival function $e^{-\beta x^{d-1}}$.
\end{itemize}
\end{theorem}
\begin{proof} 
Fix $x>0$ and observe that in this example
$$\alpha_t(x)={t^2\over 2}\int\limits_{\S^{d-1}}\int\limits_{\S^{d-1}}\1(\dist(y_1,y_2)\leq xt^{-\gamma})\;\dint y_1\,\dint y_2.$$ 
For fixed $y_2\in\S^{d-1}$ the inner integral is the  $(d-1)$-dimensional volume of the intersection of $\S^{d-1}$ with the $d$-dimensional ball $B_{xt^{-\gamma}}^d(y_2)$. The geometric structure of $\S^{d-1}$ implies that 
$$V_{d-1}(\S^{d-1}\cap B_{xt^{-\gamma}}^d(y_2))=\kappa_{d-1}(xt^{-\gamma})^{d-1}+\frac{(d-1)\kappa_{d-1}}{2}(xt^{-\gamma})^d+O(t^{-\gamma(d+1)})$$ as $t\rightarrow\infty$, independently of $y_2$; see Section 3 in \cite{R02}.
Taking $\gamma=2/(d-1)$, we conclude the asymptotic expansion
$$\alpha_t(x)=\frac{d}{2}\kappa_d\kappa_{d-1}x^{d-1}+\frac{d(d-1)\kappa_d\kappa_{d-1}}{4}x^dt^{-2/(d-1)}+O(t^{-4/(d-1)})$$ as $t\rightarrow\infty$. For
$$r_t(x)=\sup_{y_2\in\S^{d-1}}t\int\limits_{\S^{d-1}}\1(\dist(y_1,y_2)\leq xt^{-\gamma})\;\dint y_1=t V_{d-1}(\S^{d-1}\cap B_{xt^{-\gamma}}^d(y))$$ with an arbitrary $y\in\S^{d-1}$ we similarly have
$$r_t(x)=\kappa_{d-1}x^{d-1}t^{-1}+\frac{(d-1)\kappa_{d-1}}{2}x^dt^{-(d+1)/(d-1)}+O(t^{-(d+3)/(d-1)})\rightarrow 0$$
as $t\rightarrow\infty$.
The result is now found by application of Theorem \ref{thm:general}.
\end{proof}

\paragraph{4. Edges in a Gilbert graph.} Let $\eta_t$ be the restriction of a stationary Poisson point process on $\R^d$ with intensity $t\geq 1$ to a compact convex set $W\subset\R^d$ ($d\geq 1$) with positive volume $V_d(W)>0$. The \textit{Gilbert graph} or the \textit{random geometric graph} is constructed by connecting two points of $\eta_t$ by an edge if and only if their distance is smaller than a prescribed bound $\delta>0$; see the monograph \cite{Penrose2003} for an exhaustive reference and \cite{Schmidt} for a closely related recent work. In the following, we assume that the threshold $\delta$ also depends on the intensity parameter $t$ and write $\delta_t$ for this reason. Define
$$\xi_t=\{\dist(y_1,y_2):(y_1,y_2)\in\eta_{t,\neq}^2\;{\rm with}\;\dist(y_1,y_2)\leq \delta_t\}$$
and denote by $G_t^{(m)}$ the distance to the origin of the $m$-th smallest element of $\xi_t$.
\begin{theorem}\label{thm:gilbertraph}
Assume that $\lim\limits_{t\rightarrow\infty} t^{2/d}\delta_t=\infty$ and let $\beta=\frac{\kappa_d}{2}V_d(W)$.
\begin{itemize}
 \item[a)] As $t\rightarrow\infty$ the re-scaled point processes $t^{2/d}\xi_t$ converge in distribution to a Poisson point process on $\R_+$ with intensity measure $$B\mapsto\beta d\int\limits_Bu^{d-1}\;\dint u,\qquad B\subset \R_+\;{\rm Borel}.$$ 
 \item[b)] For every $x>0$ there are constants $C>0$ and $t_0>1$ depending on $x$, $W$, $d$ and $(\delta_t)_{t\geq 1}$ such that
$$\left| \P(t^{2/d}G_t^{(m)}>x) - e^{-\beta x^d}\sum_{i=0}^{m-1}\frac{(\beta x^d)^i}{i!}\right|\leq Ct^{-\min\{2/d,1/2\}}$$ for $t\geq t_0$ and $m\geq 1$. In particular, the distribution of the re-scaled shortest edge length $t^{2/d}G_t^{(1)}$ converges as $t\rightarrow\infty$ to a Weibull distribution with survival function $e^{-\beta x^d}$.
\end{itemize}
\end{theorem}
\begin{proof}
The assumption $\lim\limits_{t\rightarrow\infty} t^{2/d}\delta_t=\infty$ ensures that for every $x>0$ there is a constant $t_0\geq 1$ such that $xt^{-\gamma}\leq \delta_t$ for all $t\geq t_0$. For such $x$ and $t\geq t_0$ we have
\begin{eqnarray*}
\alpha_t(x) & = & \frac{t^2}{2}\int\limits_{W}\int\limits_W\1(\dist(y_1,y_2)\leq xt^{-\gamma})\;\dint y_1 \; \dint y_2 = \frac{t^2}{2}\int\limits_{W} V_d(W\cap B^d_{xt^{-\gamma}}(y_2))\;\dint y_2\\
& = & \frac{t^2}{2}\int\limits_{\R^d} V_d(W\cap B^d_{xt^{-\gamma}}(y_2))\;\dint y_2 - \frac{t^2}{2}\int\limits_{\R^d\setminus W} V_d(W\cap B^d_{xt^{-\gamma}}(y_2))\;\dint y_2.
\end{eqnarray*}
From Theorem 5.2.1 in \cite{SW} (see Eq. (5.14) in particular), it follows that
\begin{equation}\label{eq:Ex4summand1}
\frac{t^2}{2}\int\limits_{\R^d} V_d(W\cap B^d_{xt^{-\gamma}}(y_2))\;\dint y_2=\frac{t^2}{2} V_d(W)V_d(B^d_{xt^{-\gamma}})=\frac{\kappa_d }{2}V_d(W)t^2 (xt^{-\gamma})^d.
\end{equation}
By the Steiner formula (see Eq. (14.5) in \cite{SW}), one has
\begin{eqnarray*}
&& \frac{t^2}{2}\int\limits_{\R^d\setminus W} V_d(W\cap B^d_{xt^{-\gamma}}(y_2))\;\dint y_2\\
& \leq & \frac{\kappa_d}{2} t^2 (xt^{-\gamma})^d V_d(\{y\in \R^d\setminus W: \dist(y,W)\leq xt^{-\gamma}\})\\
& = & \frac{\kappa_d}{2} t^2 (xt^{-\gamma})^d \sum_{j=0}^{d-1} \kappa_{d-j}V_j(W)(xt^{-\gamma})^{d-j}
\end{eqnarray*}
so that $\alpha_t(x)$ is dominated by \eqref{eq:Ex4summand1} and its asymptotic behavior is given by $\frac{\kappa_d}{2}V_d(W)t^2(xt^{-\gamma})^d$. Choosing $\gamma=2/d$ yields
$$\alpha(x)=\lim_{t\rightarrow\infty}\alpha_t(x) = \frac{\kappa_d }{2}V_d(W)x^d=\beta x^d$$
and
$$|\alpha(x)-\alpha_t(x)|\leq \frac{\kappa_d}{2}  \sum_{j=0}^{d-1} \kappa_{d-j}V_j(W)x^{2d-j}t^{-2+2j/d}.$$
Moreover, one has
$$r_t(x)=\sup_{y_2\in W}t\int\limits_W\1(\dist(y_1,y_2)\leq xt^{-\gamma})\;\dint y_1\leq t\kappa_d (xt^{-\gamma})^d = \kappa_d x^d t^{-1}.$$
The assertion is now a direct consequence of Theorem \ref{thm:general}.
\end{proof} 

\begin{remark}\rm
As $t\rightarrow\infty$, the shortest edge of the Gilbert graph is the same as the shortest edge of the so-called {\em Delaunay graph} (cf. \cite{Penrose2003,SW} for background material on Delaunay graphs or tessellations). Hence, the length of the shortest edge in the Delaunay graph enjoys the same asymptotic behavior as $G_t^{(1)}$ considered in Theorem \ref{thm:gilbertraph}.
\end{remark}

\paragraph{5. Small simplices generated by Poisson points.} Denote by $\eta_t$ the restriction of a stationary Poisson point process of intensity $t\geq 1$ to a compact convex set $W\subset\R^d$ ($d\geq 1$) with volume $V_d(W)>0$. Consider the family of all $d$-dimensional simplices that can be formed by any $d+1$ distinct points of $\eta_t$ and denote by $S_t^{(m)}$ the $m$-th smallest volume of them. The sequence $(S_t^{(m)})_{m\geq 1}$ forms a point process $\xi_t $, i.e.
$$\xi_t=\{V_d([y_1,\ldots,y_{d+1}]):(y_1,\ldots,y_{d+1})\in\eta_{t,\neq}^{d+1}\},$$ where $[y_1,\ldots,y_{d+1}]$ stands for the simplex with vertices $y_1,\ldots,y_{d+1}$. A similar problem in the special planar set-up has been studied in \cite{GJ}, where a fixed number of points in $W$ that tends to infinity was used in place of the Poisson point process. In fact, in \cite{GJ} the authors showed Poisson point process convergence for the re-scaled order statistics. From this, a de-poissonized version of the planar case of our Theorem \ref{thm:simplices} below can be derived. Moreover, we like to point out that for $d=2$ we have the simple expression $\beta=2V_2(W)^2$ for the parameter in Theorem \ref{thm:simplices} below thanks to an integral-geometric formula due to Crofton; cf. \cite[Eq. (8.58)]{SW}.
\begin{theorem}\label{thm:simplices}
Define $$\beta={d\kappa_d\over d+1}\int\limits_{[W]_{d-1}}V_{d-1}(W\cap H)^{d+1}\; \dint H.$$ Then the point processes $t^{d+1}\xi_t $ converge in distribution as $t\rightarrow\infty$ to a homogeneous Poisson point process on $\R_+$ with intensity $\beta$, so that for all $m\geq 1$ and $x>0$, $$\lim_{t\rightarrow\infty}\P(t^{d+1}S_t^{(m)}>x)=e^{-\beta x}\sum_{i=0}^{m-1}{(\beta x)^i\over i!}.$$ In particular, the re-scaled smallest simplex volume $t^{d+1}S_t^{(1)}$ is asymptotically exponentially distributed with parameter $\beta$.
\end{theorem}
\begin{proof}
In this example, we have
$$\alpha_t(x)={t^{d+1}\over(d+1)!}\int\limits_W\ldots\int\limits_W\1(V_d([y_1,\ldots,y_{d+1}])\leq xt^{-\gamma})\;\dint y_{d+1}\ldots \dint y_1.$$
For fixed $y_1,\ldots,y_d\in W$ in general position let $H(y_1,\ldots,y_d)$ be the unique hyperplane through these points and put 
$$H(y_1,\ldots,y_d)_r:=\{y\in\R^d:\dist(y,H(y_1,\ldots,y_d))\leq r\}$$
for $r>0$. With the abbreviation $\varrho=\varrho(y_1,\ldots,y_d):=dxt^{-\gamma}/V_{d-1}([y_1,\ldots,y_d])$ and the affine Blaschke-Petkantschin formula \cite[Thm. 7.2.5]{SW}, $\alpha_t(x)$ can be re-written as 
\begin{eqnarray*}
\alpha_t(x) &=& {t^{d+1}\over(d+1)!}\int\limits_W\ldots\int\limits_W V_d(H(y_1,\ldots, y_d)_\varrho\cap W)\;\dint y_d\ldots \dint y_1\\
& = & \frac{\kappa_d t^{d+1}}{2(d+1)}\int\limits_{[W]_{d-1}}\int\limits_{H\cap W}\hdots\int\limits_{H\cap W}V_d(H(y_1,\hdots, y_d)_\varrho\cap W)\\
& & \hspace{3cm}\times V_{d-1}([y_1,\hdots, y_d])\;\dint y_d\hdots \dint y_1\dint H.
\end{eqnarray*}
Using the fact that
\begin{equation}\label{eq:asymptoticSimplices}
\lim_{\varrho \rightarrow 0}\frac{V_d(H(y_1,\hdots, y_d)_\varrho\cap W)}{\varrho}=2 V_{d-1}(H(y_1,\hdots, y_d)\cap W)
\end{equation}
and choosing $\gamma=d+1$, we find
\begin{eqnarray*}
&&\lim_{t\rightarrow\infty} t^{d+1}V_d(H(y_1,\hdots, y_d)_\varrho\cap W)\\
 & = & \lim_{t\rightarrow\infty} \frac{V_d(H(y_1,\hdots, y_d)_{dxt^{-(d+1)}/V_{d-1}([y_1,\hdots, y_d])}\cap W)}{t^{-(d+1)}}\\
& = & 2 d x V_{d-1}(H(y_1,\hdots, y_d)\cap W) V_{d-1}([y_1,\hdots, y_d])^{-1}.
\end{eqnarray*}
From the dominated convergence theorem, it finally follows that
\begin{eqnarray*}
\alpha & = & \lim_{t\rightarrow \infty} \alpha_t(x) =\frac{d\kappa_d x}{d+1}\int\limits_{[W]_{d-1}}\int\limits_{H\cap W}\hdots\int\limits_{H\cap W} V_{d-1}(W\cap H) \; \dint y_d\hdots \dint y_1 \; \dint H\\ 
& = & \frac{d\kappa_d x}{d+1} \int\limits_{[W]_{d-1}} V_{d-1}(W\cap H)^{d+1} \; \dint H=\beta x.
\end{eqnarray*}
By a similar computation, one sees that the condition \eqref{eq:condition} is also satisfied and that the assertion is a consequence of Theorem \ref{thm:general}.
\end{proof}
\begin{remark}\rm
Although Theorem \ref{thm:general} b) delivers an exact rate of convergence, we cannot provide an explicit rate here. This is due to the fact that the exact asymptotic behavior in (\ref{eq:asymptoticSimplices}) depends in a delicate way on the smoothness of the boundary of $W$. A similar comment also applies to the next example.
\end{remark}

\paragraph{6. Small simplices generated by Poisson hyperplanes.}
Let $\eta_t$ be the restriction of a stationary and isotropic Poisson point process of hyperplanes of intensity $t\geq 1$ to $[W]_{d-1}$, where $W$ is a compact convex set with $W\subset\R^d$ ($d\geq 2$) with $V_d(W)>0$. Any $d+1$ different hyperplanes of $\eta_t$ generate a random simplex in $\R^d$ almost surely. Let $[H_1,\hdots,H_{d+1}]$ be the simplex generated by the hyperplanes $H_1,\hdots,H_{d+1}$, put
$$\xi_t=\{V_d([H_1,\hdots,H_{d+1}]):(H_1,\hdots,H_{d+1})\in\eta_{t,\neq}^{d+1}\;{\rm with}\; [H_1,\hdots,H_{d+1}]\subset W\}$$
and let $T_t^{(m)}$ be the $m$-th smallest simplex volume. For fixed hyperplanes $H_1,\hdots, H_d$, we denote by $H_{u,\delta}$ the hyperplane with unit normal vector $u\in{\S}^{d-1}$ and distance $\delta>0$ to the almost surely uniquely determined intersection point of $H_1,\hdots,H_d$.

\begin{theorem}
Define $$\beta=\frac{1}{(d+1)!}\int\limits_{[W]_{d-1}}\hdots\int\limits_{[W]_{d-1}}\int\limits_{\S^{d-1}}V_d([H_1,\hdots,H_d,H_{u,1}])^{-1/d} \; \dint u \; \dint H_d\hdots \dint H_1.$$ Then as $t\rightarrow\infty$ the re-scaled point processes $t^{d(d+1)}\xi_t $ converge in distribution to a Poisson point process whose intensity measure is given by $$B\mapsto\beta d^{-1}\int\limits_Bu^{(1-d)/d}\;\dint u,\qquad B\subset \R_+\;{\rm Borel},$$ whence for $m\geq 1$ and $x>0$ it holds that
$$\lim_{t\rightarrow\infty} \P(t^{d(d+1)}T_t^{(m)}>x)=e^{-\beta x^{1/d}}\sum_{i=0}^{m-1} \frac{(\beta x^{1/d})^i}{i!}.$$ In particular, $t^{d(d+1)}T_t^{(1)}$ converges to a Weibull distributed random variable with survival function $e^{-\beta x^{1/d}}$.
\end{theorem}
\begin{proof}
We have 
\begin{eqnarray*}
\alpha_t(x) & = & \frac{t^{d+1}}{(d+1)!}\int\limits_{[W]_{d-1}} \!\!\!\!\! \hdots \int\limits_{[W]_{d-1}} \!\!\!\!\! \1([H_1,\hdots,H_{d+1}]\subset W)\\ && \hspace{3.65cm} \times\1(V_d([H_1,\hdots,H_{d+1}])\leq xt^{-\gamma}) \; \dint H_{d+1}\hdots \dint H_1
\end{eqnarray*}
in this example. For fixed hyperplanes $H_1,\hdots,H_d$, we identify $H_{d+1}$ with the pair $(u,\delta)\in{\S}^{d-1}\times [0,\infty)$ and write $H_{u,\delta}$ instead of $H_{d+1}$. Since 
$$V_d([H_1,\hdots,H_d,H_{u,\delta}])=\delta^dV_d([H_1,\hdots,H_d,H_{u,1}])$$
and $[H_1,\hdots,H_d,H_{u,\delta}]\subset W$ whenever $\delta$ is small enough, it holds that
\begin{eqnarray*}
&&\int\limits_0^\infty  \1([H_1,\hdots,H_d,H_{u,\delta}]\subset W) \; \1(V_d([H_1,\hdots,H_d,H_{u,\delta}])\leq xt^{-\gamma}) \; \dint \delta\\ & = & \frac{(xt^{-\gamma})^{1/d}}{V_d([H_1,\hdots,H_d,H_{u,1}])^{1/d}}
\end{eqnarray*}
if $t$ is sufficiently large. By the choice $\gamma=d(d+1)$ and the dominated convergence theorem, we obtain
\begin{eqnarray*}
\alpha & = & \lim_{t\rightarrow\infty}\alpha_t(x)\\
& = & \frac{x^{1/d}}{(d+1)!}\int\limits_{[W]_{d-1}}\hdots\int\limits_{[W]_{d-1}}\int\limits_{\S^{d-1}}V_d([H_1,\hdots,H_d,H_{u,1}])^{-1/d} \; \dint u \; \dint H_d\hdots \dint H_1\\
& = & \beta x^{1/d}.
\end{eqnarray*}
Condition \eqref{eq:condition} can be checked in a similar way, hence, Theorem \ref{thm:general} can be applied and completes the proof.
\end{proof}

The hyperplanes of $\eta_t$ partition the space into random polytopes (called cells) and generate this way a tessellation of ${\R}^d$, the so-called {\em Poisson hyperplane tessellation}, which is one of the standard models considered in stochastic geometry; see \cite{SW} and the references therein. In the planar case $d=2$, the smallest triangle generated by the Poisson lines (i.e. $1$-dimensional hyperplanes in ${\R}^2$) cannot be hit by other lines of $\eta_t$, since otherwise there would be an even smaller triangle (note that in higher dimensions this argument fails). Hence, the smallest triangle is also a cell of the tessellation and its area enjoys the following asymptotic behavior.

\begin{corol}\label{cor:smallcells}
Let $\Delta_t^{\rm min}$ be the smallest triangular cell that is included in $W$ of a tessellation generated by a stationary and isotropic Poisson line process of intensity $t\geq 1$. Then $(t^6V_2(\Delta_t^{\rm min}))_{t\geq 1}$ converges as $t\rightarrow\infty$ in distribution to a Weibull distributed random variable with survival function $e^{-\beta\sqrt{x}}$.
\end{corol}

\begin{remark}\rm Using heuristic arguments, it has been argued in \cite{Miles} that small cells of tessellations generated by a stationary and isotropic Poisson line process have a triangular shape. This way, Corollary \ref{cor:smallcells} makes a statement not only about the area of the smallest triangular cell, but also about the area of the smallest cell in general. A formal proof, however, is still missing.
\end{remark}

\section{Chaos and Poisson approximation}\label{secCHAOS}

The framework is a non-atomic Borel measurable space $(\Y,\cY)$ with a $\sigma$-finite measure $\lambda$. In what follows, $\eta$ is a \textit{Poisson point process} on $\Y$ with \textit{intensity measure} $\lambda$. This is to say, $\eta$ is a collection of random variables defined on some probability space $(\Omega,{\cal F},\P)$, indexed by the elements of $\cY_\lambda=\{B\in\cY:\lambda(B)<\infty\}$ such that
\begin{itemize}
 \item[(i)] for disjoint sets $A,B\in\cY_\lambda$, $\eta(A)$ and $\eta(B)$ are independent;
 \item[(ii)] $\eta(B)$ is Poisson distributed with mean $\lambda(B)$ for any $B\in\cY_\lambda$.
\end{itemize}
We also write $\hat{\eta}(B)=\eta(B)-\lambda(B)$ for $B\in\cY_\lambda$ and $\{\hat{\eta}(B):B\in\cY_\lambda\}$ for the \textit{compensated Poisson point process}. As usual in point process theory, we shall identify $\eta$ with its support and write $y\in\eta$ to indicate that $y\in\Y$ is charged by $\eta$. Similarly, we write $(y_1,\ldots,y_k)\in\eta_{\neq}^k$ for $k\geq 1$ to say that $y_1,\ldots,y_k$ are distinct points in $\Y$ with $y_i\in\eta$ for $i=1,\ldots,k$.

Given two integers $p,q\geq 1$, write $L^p(\lambda^q)$ for the family of functions $h:\Y^q\rightarrow\R$ such that $||h||_{L^p(\lambda^q)}:=(\int_{\Y^q}|h|^p\;\dint\lambda^q)^{1/p}<\infty$ and $L_{\rm sym}^p(\lambda^q)$ for the subspace of $L^p(\lambda^q)$ consisting of functions that are invariant under permutation of the $q$ arguments, so-called \textit{symmetric functions}. The standard scalar product in $L^2(\lambda^q)$ is denoted by $\langle\cdot,\cdot\rangle_{L^2(\lambda^q)}$.

For a deterministic function $h\in L^2(\lambda)$ we write $I_1(h)$ for the \textit{Wiener-It\^o integral} of $h$ and for every $q\geq 2$ and $h\in L^2_{\rm sym}(\lambda^q)$ we indicate by $I_q(h)$ the \textit{multiple Wiener-It\^o integral of order $q$} of $h$ with respect to the compensated Poisson point process $\hat{\eta}$; cf. \cite{LastPenrose2011,NualartVives90,PeccatiTaqqu2010}. These stochastic integrals are centered (i.e. $\E I_q(h)=0$) and satisfy the isometry relation
\begin{equation}\label{eq:isometry}
\E[I_m(g)I_n(h)]=\langle g,h\rangle_{L^2(\lambda^n)}\1(n=m)
\end{equation}
for any integers $m,n\geq 1$ and every $g\in L_{\rm sym}^2(\lambda^m)$ and $h\in L_{\rm sym}^2(\lambda^n)$. The Hilbert space $\{I_q(h):h\in L_{\rm sym}^2(\lambda^q)\}$, $q\geq 1,$ is called the \textit{$q$-th Wiener-It\^o chaos} associated with $\eta$.

It is one of the crucial features of a Poisson point process that every $F\in L^2(\P_{\eta})$ (here and below $\P_\eta$ stands for the distribution of $\eta$) can be decomposed into its chaotic components, which is to say that $F$ may be written as 
\begin{equation}\label{eq:chaos}
F=\E F+\sum_{q=1}^\infty I_q(h_q),
\end{equation}
where the series converges in $L^2(\P_\eta)$ and for each $q\geq 1$, $h_q$ is an element in $L_{\rm sym}^2(\lambda^q)$. The representation (\ref{eq:chaos}) is called \textit{Wiener-It\^o chaos decomposition} of $F$ with \textit{kernels} $h_q$. Note in particular that (\ref{eq:chaos}) combined with the isometry (\ref{eq:isometry}) implies the variance formula
\begin{equation}\label{eq:varianceGeneral}
{\rm Var}\;F=\sum_{q=1}^\infty q!||h_q||_{L^2(\lambda^q)}^2.
\end{equation}
We will later need the chaos decomposition of a special class of Poisson functionals introduced in \cite{RS11}. Let $f\in L_{\rm sym}^1(\lambda^k)$ and define
\begin{equation}\label{eq:geomustat}
U=\frac{1}{k!}\sum_{(y_1,\ldots,y_k)\in\eta_{\neq}^k}f(y_1,\ldots,y_k)
\end{equation}
for some fixed $k\geq 1$. If the functional $U$ satisfies $U\in L^2(\P_\eta)$, it is called a \textit{Poisson U-statistic} and from \cite{RS11} we know that the chaos decomposition of $U$ is finite and given by $U=\E U+\sum_{q=1}^kI_q(h_q)$ with $\E U=\int_{\Y^k}f\;\dint\lambda^k$ by the Campbell Theorem for point processes \cite[Thm. 3.1.2]{SW} and with 
\begin{equation}\label{eq:kernelsUstatistic}
h_q(y_1,\hdots,y_q)=\frac{1}{q!(k-q)!}\int\limits_{\Y^{k-q}} f(y_1,\hdots,y_q,\hat{y}_1,\hdots,\hat{y}_{k-q})\;\lambda^{k-q}(\dint(\hat{y}_1,\hdots,\hat{y}_{k-q})) 
\end{equation}
for $q=1,\ldots,k$. (Thanks to the structure of $U$ we have that $h_q\equiv 0$ for $q\geq k+1$.)

In the proof of Theorem \ref{thm:general} we make use of two Malliavin-type operators on the Poisson space. We will briefly recall their definitions and refer to \cite{NualartVives90,Peccatietal2010} for further details. We denote by ${\rm dom}\; D$ the set of all $F\in L^2(\P_\eta)$ with chaos decomposition (\ref{eq:chaos}) satisfying $\sum_{q=1}^\infty qq!||h_q||_{L^2(\lambda^q)}^2<\infty$ and define for $F\in{\rm dom}\;D$ the random function $\Y\ni y\mapsto D_yF$ by $$D_yF=\sum_{q=1}^\infty qI_{q-1}(h_q(y,\cdot)).$$ The operator $D$ is called the \textit{Malliavin derivative} and has an intuitive interpretation as difference operator. In fact, it holds that $$D_yF(\eta)=F(\eta+\delta_y)-F(\eta),$$ where $\delta_y$ stands for the unit mass Dirac measure at $y\in\Y$ cf. Lemma 2.5 in \cite{NualartVives90}. Besides $D$ we need the \textit{pseudo-inverse $L^{-1}$ of the Ornstein-Uhlenbeck generator}. For centered $F\in L^2(\P_\eta)$ as in (\ref{eq:chaos}) we put $$L^{-1}F=-\sum_{q=1}^\infty{1\over q}I_q(h_q).$$ Moreover, for a not necessarily centered Poisson functional $F\in L^2(\P_\eta)$ we define $L^{-1}F=L^{-1}(F-\E F)$ by convention. (Note that for $F=U$ as in (\ref{eq:geomustat}) the sums in the definitions of $D_yU$ and $L^{-1}U$ are in fact finite.)

We are now prepared to rephrase one of the main findings of the paper \cite{Peccati11}, Theorem 3.1 ibidem. It has been obtained by a combination of the Chen-Stein method for Poisson approximation and the Malliavin calculus of variations on the Poisson space. To state the result, we denote by $d_{\rm TV}(X,Y)$ the total variation distance of two non-negative integer-valued random variables $X$ and $Y$, i.e. $$d_{\rm TV}(X,Y)={1\over 2}\sum_{n=0}^\infty|\P(X=n)-\P(Y=n)|.$$
\begin{proposition}\label{prop:peccati}
Let $F\in L^2(\P_\eta)$ be such that $F$ belongs to ${\rm dom}\;D$, takes values only in $\{0,1,2,\ldots\}$ and satisfies $\E F=u>0$. Furthermore, let $Po(v)$ be a Poisson distributed random variable with mean $v>0$. Then
\begin{equation}\label{eq:totalvariation}
\begin{split}
d_{\rm TV}(F,Po(v)) &\leq  |u-v|+{1-e^{-u}\over u}\,\E|u-\langle DF,-DL^{-1}F\rangle_{L^2(\lambda)}|\\
 &\;\;\;\;\;\;\;+{1-e^{-u}\over u^2}\,\E\int\limits_{\Y}|D_yF(D_yF-1)D_yL^{-1}F|\;\lambda(\dint y).
\end{split}
\end{equation}
\end{proposition}
We note that if $d_{\rm TV}(F_t,Po(v))\rightarrow 0$ as $t\rightarrow\infty$ for a family $(F_t)_{t\geq 1}$ of Poisson functionals as in Proposition \ref{prop:peccati}, $(F_t)_{t\geq 1}$ converges in distribution to $Po(v)$; cf. Prop. 3.3 in \cite{Peccati11}. This holds, because the topology induced by $d_{\rm TV}$ on the class of probability distributions on the non-negative integers is strictly finer than the topology induced by convergence in distribution.

In order to evaluate the right-hand side in \eqref{eq:totalvariation}, we need the so-called {\em product formula for multiple Wiener-It\^o integrals} (see \cite{PeccatiTaqqu2010,Surgailis1984}). Before stating it, we introduce some further notation. Let $n_1,\hdots,n_m\in\N$ and let $f^{(i)}\in L_{\rm sym}^2(\lambda^{n_i})$ for $i=1,\hdots,m$, where $m\geq 1$ is a fixed integer. We denote the arguments of $f^{(i)}$ by $y^{(i)}_1,\hdots,y^{(i)}_{n_i}$ and let $\otimes_{i=1}^m f^{(i)}: \Y^{\sum_{i=1}^m n_i}\rightarrow\R$ be given by $$\otimes_{i=1}^m f^{(i)}(y_1^{(1)},\hdots,y_{n_m}^{(m)})=\prod_{i=1}^m f^{(i)}(y_1^{(i)},\hdots,y_{n_i}^{(i)}).$$ 
By $\Pi(n_1,\hdots,n_m)$ we denote the set of all partitions $\pi$ of the variables $$y_1^{(1)},\hdots,y_{n_1}^{(1)},\hdots,y_1^{(m)},\hdots,y_{n_m}^{(m)}$$ having the property that two variables with the same upper index are always in two different elements of $\pi$. We shall write $|\pi|$ for the number of elements of $\pi$. Let also $\Pi_{\geq 2}(n_1,\hdots,n_m)$ be the collection of all those partitions $\pi\in\Pi(n_1,\hdots,n_m)$ where every element of $\pi$ includes at least two variables.

For every $\pi\in\Pi(n_1,\hdots,n_m)$, the function $\left(\otimes_{i=1}^m f^{(i)}\right)_\pi: \Y^{|\pi|}\rightarrow\R$ is given by replacing all variables of $\otimes_{i=1}^m f^{(i)}$ that belong to the same element of $\pi$ by a new variable. For example, if $$(f^{(1)}\otimes f^{(2)}\otimes f^{(3)})(y_1^{(1)}, y_2^{(1)},y_1^{(2)},y_2^{(2)},y_1^{(3)})=f^{(1)}(y_1^{(1)}, y_2^{(1)})f^{(2)}(y_1^{(2)},y_2^{(2)}) f^{(3)}(y_1^{(3)})$$
and if $\pi=\{(y_1^{(1)},y_1^{(2)},y_1^{(3)}), (y_2^{(1)},y_2^{(2)})\}$, then 
$$(f^{(1)}\otimes f^{(2)}\otimes f^{(3)})_\pi(\hat{y}_1,\hat{y}_2)=f^{(1)}(\hat{y}_1,\hat{y}_2)f^{(2)}(\hat{y}_1,\hat{y}_2)f^{(3)}(\hat{y}_1).$$
We refer the reader to \cite{PeccatiTaqqu2010} for further details. 

\begin{proposition}\label{prop:productformula}
Let for $i=1,\hdots,m$, $f^{(i)}:\Y^{n_i}\rightarrow \R$ be bounded, non-negative and symmetric such that $\int_{\Y^{|\pi|}}(\otimes_{i=1}^m f^{(i)})_\pi\;\dint\lambda^{|\pi|}<\infty$ for all $\pi\in\Pi(n_1,\ldots,n_m)$. Then
\begin{equation}\label{eq:productformula}
\E \prod_{i=1}^m I_{n_i}(f^{(i)})=\sum_{\pi\in\Pi_{\geq 2}(n_1,\hdots,n_m)}\;\int\limits_{\Y^{|\pi|}} \left(\otimes_{i=1}^{m} f^{(i)}\right)_\pi \dint\lambda^{|\pi|}.
\end{equation} 
\end{proposition}
\begin{remark}\rm The product formula \eqref{eq:productformula} is stated as Corollary 7.2 in \cite{PeccatiTaqqu2010} for the case that the $f^{(i)}$ are so-called simple functions, but can be extended to our setting. In this generality, it is also a corollary of the main finding in \cite{Surgailis1984}. 
\end{remark}

To simplify the notation and the arguments in the next section, let $\tilde{\Pi}(n_1,\hdots,n_m)$, respectively $\tilde{\Pi}_{\geq 2}(n_1,\hdots,n_m)$, be the set of all partitions $\pi\in\Pi(n_1,\hdots,n_m)$, respectively $\pi\in\Pi_{\geq 2}(n_1,\hdots,n_m)$, such that for every partition of $\{1,\hdots,m\}$ into two disjoint sets $M_1$ and $M_2$ there is an element of $\pi$ including variables with upper indexes $i_1\in M_1$ and $i_2\in M_2$.

\section{Proof of Theorem \ref{thm:general}}\label{secPOISSONAPPROXIMATION}

The basic idea to derive Theorem \ref{thm:general} is to consider the Poisson U-statistics $U_{A_t}$ given by
$$U_{A_t}=\frac{1}{k!}\sum_{(y_1,\hdots,y_k)\in\eta^k_{t,\neq}} \1(f(y_1,\hdots,y_k)\in A_t)$$
for a family $(A_t)_{t\geq 1}$ of Borel sets $A_t\subset\R_+$ with $x_{max}=\sup\limits_{t\geq 1}\sup\limits_{x\in A_t}x<\infty$ and define $h_1,\hdots,h_k$ by \eqref{eq:kernelsUstatistic}. We notice that for $q=1,\ldots,k$, $h_q\in L^2(\lambda_t^q)$. Indeed,
\begin{eqnarray*}
||h_q||_{L^2(\lambda^q)}^2 \!\!\!\!\!\! &  =& \!\!\!\! \frac{1}{(q!(k-q)!)^2}\int\limits_{\Y^q}\int\limits_{\Y^{k-q}} \! \1(f(y_1,\ldots,y_q,\hat{y}_1,\ldots,\hat{y}_{k-q})\in A_t)\; \lambda^{k-q}(\dint(\hat{y}_1,\ldots,\hat{y}_{k-q}))\\
& &  \times \int\limits_{\Y^{k-q}}\1(f(y_1,\ldots,y_q,\tilde{y}_1,\ldots,\tilde{y}_{k-q})\in A_t)\;\lambda^{k-q}(\dint(\tilde{y}_1,\ldots,\tilde{y}_{k-q}))\\
& & \qquad\lambda^q(\dint(y_1,\ldots,y_q))\\
&\leq & \frac{k!}{(q!(k-q)!)^2}\alpha_t(x_{max})r_t(x_{max})<\infty.
\end{eqnarray*}
Combining this with Campbell's theorem for point processes and some combinatorial arguments yields that $U_{A_t}\in L^2(\P_{\eta_t})$ and that \eqref{eq:varianceGeneral} holds. Hence, $U_{A_t}$ has finite Wiener-It\^o chaos decomposition with kernels $h_1,\ldots,h_k$.

\begin{proposition}\label{prop:PoissonApprox}
Assume that there is a constant $0<\sigma<\infty$ with 
$$\sigma_t=\frac{1}{k!}\int\limits_{\Y^k} \1(f(y_1,\hdots,y_k)\in A_t)\;\lambda_t^k(\dint(y_1,\hdots,y_k))\rightarrow \sigma \ \text{ as }\ t\rightarrow\infty$$
and suppose that
$$ \rho_t=\sup_{\substack{y_1,\hdots,y_{k-j}\in\Y\\ 1\leq j\leq k-1}}\lambda_t^j(\{(\hat{y}_1,\hdots,\hat{y}_j)\in\Y^j: f(\hat{y}_1,\hdots, \hat{y}_j,y_1,\hdots,y_{k-j})\in A_t\})\rightarrow 0$$
as $t\rightarrow\infty$. Then
$$d_{\rm TV}(U_{A_t},Po(\sigma))\leq |\sigma-\sigma_t| + C_k \frac{1-e^{-\sigma_t}}{\sigma_t} \left(1+\frac{1}{\sigma_t}\right) \sqrt{\sigma_t\left(\rho_t+\rho_t^3\right)}$$
with a constant $C_k$ only depending on $k$.
\end{proposition}
To prepare for the proof of Proposition \ref{prop:PoissonApprox} we need the following lemma.
\begin{lemma}\label{lem:intkernels}
Let $\ell\geq 2$, let $1\leq n_i\leq k$ for $i=1,\hdots,\ell$ and let $\pi\in \tilde{\Pi}(n_1,\hdots,n_\ell)$. Then 
$$\int\limits_{\Y^{|\pi|}} \left(\otimes_{i=1}^\ell h_{n_i}\right)_\pi \dint\lambda_t^{|\pi|}=\frac{1}{(k!)^{\ell-1}}\sigma_t$$
if and only if $n_1=\hdots=n_\ell=k$ and $|\pi|=k$. Otherwise,
\begin{eqnarray*}
\int\limits_{\Y^{|\pi|}} \left(\otimes_{i=1}^\ell h_{n_i}\right)_\pi \dint\lambda_t^{|\pi|} & \leq & k!\left(\prod_{i=1}^\ell \frac{1}{n_i!(k-n_i)!}\right) \sigma_t \left(\rho_t+\rho_t^{\ell-1}\right)\\
& \leq & k! \sigma_t \left(\rho_t+\rho_t^{\ell-1}\right).
\end{eqnarray*}
\end{lemma} 
\begin{proof}
Recall \eqref{eq:kernelsUstatistic}, fix $\pi\in\tilde{\Pi}(n_1,\hdots,n_\ell)$ and construct another partition $\pi^*$ from $\pi$ by adding all variables $y_j^{(i)}$ over which the integration runs in the definition of $h_{n_i}$ as elements $\{y_j^{(i)}\}$ to $\pi$ for $i=1,\hdots,\ell$. Then
\begin{equation}\label{eq:intkernels}
\begin{split}
&\int\limits_{\Y^{|\pi|}} \left(\otimes_{i=1}^\ell h_{n_i}\right)_\pi \dint\lambda_t^{|\pi|}\\ &= \left(\prod_{i=1}^\ell\frac{1}{n_i!(k-n_i)!}\right) \int\limits_{\Y^{|\pi^*|}}(\otimes_{i=1}^\ell \1(f(y_1^{(i)},\hdots,y_k^{(i)})\in A_t))_{\pi^*}\;\lambda_t^{|\pi^*|}(\dint(\hat{y}_1,\ldots,\hat{y}_{|\pi^*|})),
\end{split}
\end{equation}
where $\hat{y}_1,\hdots,\hat{y}_{|\pi^*|}$ are the replacing variables in the definition of $(\hdots)_{\pi^*}$. 
If $\pi\in \tilde{\Pi}(k,\hdots,k)$ and $|\pi|=k$, the right hand side in \eqref{eq:intkernels} simplifies to
$$\frac{1}{(k!)^\ell}\int\limits_{\Y^k}\1(f(\hat{y}_1,\hdots,\hat{y}_k)\in A_t)\;\lambda_t^k(\dint(\hat{y}_1,\hdots,\hat{y}_k))=\frac{1}{(k!)^{\ell-1}}\sigma_t,$$
which proves the first claim. For the second claim, one starts with the last function of the tensor product on the right hand side in \eqref{eq:intkernels}. If all of its variables $\hat{y}_i$ also occur in other functions of the tensor product, we can bound it by $1$. Otherwise, the integration over the variables that only occur in this function yields a positive real number less or equal than $\rho_t$ by the definition of $\rho_t$ in Proposition \ref{prop:PoissonApprox}. Iterating this procedure, a power of $\rho_t$ with an exponent between $1$ and $\ell-1$ as well as the integral $\int_{\Y^k}\1(f(\hat{y}_1,\hdots,\hat{y}_k)\in A_t)\;\lambda_t^k(\dint(\hat{y}_1,\hdots,\hat{y}_k))=k!\sigma_t$ remain. This completes the proof.
\end{proof}

As a consequence of the proof of Lemma \ref{lem:intkernels} we conclude that $\int_{\Y^{|\pi|}}(\otimes_{i=1}^\ell h_{n_i})\;\dint\lambda_t^{|\pi|}<\infty$ for all $\pi\in\Pi(n_1,\ldots,n_\ell)$ with $1\leq n_i\leq k$, $i=1,\ldots,\ell$, $\ell\geq 2$. If $\pi\in\tilde{\Pi}(n_1,\ldots,n_\ell)$, this is clear from the Lemma. Otherwise, one can write the integral as product of integrals of the type considered above. This implies that the assumptions of the product formula in Proposition \ref{prop:productformula} are satisfied with $f^{(i)}=h_{n_i}$ there. This is used without further comments several times below.

\begin{proof}[Proof of Proposition \ref{prop:PoissonApprox}.] Combining the general bound \eqref{eq:totalvariation} from Proposition \ref{prop:peccati} with the triangle inequality and the Cauchy-Schwarz inequality, we obtain
\begin{equation}\label{eq:bound0}
\begin{split}
 &d_{\rm TV}(U_{A_t},Po(\sigma))\leq  |\sigma-\sigma_t|\\ & \quad+\frac{1-e^{-\sigma_t}}{\sigma_t}(\;\underbrace{|\sigma_t-\V\,U_{A_t}|}_{=:T_1}+\underbrace{\E|\V\,U_{A_t}-\langle DU_{A_t},-DL^{-1}U_{A_t}\rangle_{L^2(\lambda_t)}|}_{=:T_2}\;) \\
 & \quad+ \frac{1-e^{-\sigma_t}}{\sigma_t^2}(\;\underbrace{\E\int\limits_{\Y} \left(D_yU_{A_t}(D_yU_{A_t}-1)\right)^2\lambda_t(\dint y)}_{=:T_3}\;\underbrace{\E\int\limits_{\Y} (D_yL^{-1}U_{A_t})^2\lambda_t(\dint y)}_{=:T_4}\;)^\frac{1}{2}.
\end{split}
\end{equation}
In the following, we bound the expressions $T_1$--$T_4$ on the right-hand side in \eqref{eq:bound0} to obtain a rate of convergence. To start with $T_1$, recall \eqref{eq:varianceGeneral} and apply Lemma \ref{lem:intkernels} to conclude that
\begin{equation}\label{eq:varUat}
\V\,U_{A_t}=\sum_{q=1}^k q! ||h_q||_{L^2(\lambda_t^q)}^2= \sigma_t + R_{\V\,U_{A_t}} 
\end{equation}
with $R_{\V\,U_{A_t}}$ satisfying
$$|R_{\V\,U_{A_t}}|\leq \sum_{q=1}^{k-1}\frac{k!}{q! ((k-q)!)^2}\sigma_t \rho_t,$$
so that
\begin{equation}\label{eq:bound1}
T_1=|\sigma_t - \V\,U_{A_t}| \leq \sum_{q=1}^{k-1}\frac{k!}{q! ((k-q)!)^2}\sigma_t\rho_t.
\end{equation}
We turn now to $T_2$. By the definition of the Malliavin operators $D$ and $L^{-1}$, the triangle inequality and the Cauchy-Schwarz inequality, we obtain
\begin{eqnarray*}
&&\E |\V\,U_{A_t}-\langle D U_{A_t},-D L^{-1}U_{A_t}\rangle_{L^2(\lambda_t)}|\\
 & = & \E\left|\sum_{q=1}^k q! ||h_q||_{L^2(\lambda_t^q)}^2 -\langle \sum_{q=1}^k q I_{q-1}(h_q(y,\cdot)),\sum_{q=1}^kI_{q-1}(h_q(y,\cdot))\rangle_{L^2(\lambda_t)}\right|\\
& \leq & \sum_{i,j=1}^k i\E|\langle I_{i-1}(h_i(y,\cdot)),I_{j-1}(h_j(y,\cdot))\rangle_{L^2(\lambda_t)}-\E \langle I_{i-1}(h_i(y,\cdot)),I_{j-1}(h_j(y,\cdot))\rangle_{L^2(\lambda_t)} |\\ &\leq&  \sum_{i,j=1}^k i \sqrt{R_{ij}}\\
\end{eqnarray*}
with
\begin{eqnarray*}
R_{ij} & = &\E\int\limits_{\Y}\int\limits_{\Y} I_{i-1}(h_i(y_1,\cdot))I_{i-1}(h_i(y_2,\cdot))I_{j-1}(h_j(y_1,\cdot))I_{j-1}(h_j(y_2,\cdot))\;\lambda_t(\dint y_1)\lambda_t(\dint y_2)\\ &&\hspace{2cm}-\left(\E \langle I_{i-1}(h_i(y,\cdot)),I_{j-1}(h_j(y,\cdot))\rangle_{L^2(\lambda_t)}\right)^2
\end{eqnarray*}
for $i,j=1,\hdots,k$. These expressions can be evaluated further using the product formula in Proposition \ref{prop:productformula} and adding the variables $y_1$ and $y_2$ to the partions similar as in \cite{RS11}. For $i\neq j$, the second term in the expression for $R_{ij}$ vanishes by the isometry relation \eqref{eq:isometry} and we obtain by the product formula for the first term expressions  only involving partitions $\pi\in\tilde{\Pi}_{\geq 2}(i,i,j,j)$. This is caused by the fact that $y_1$ and $y_2$ are fixed and all elements of a partition $\pi\in\Pi_{\geq 2}(i,i,j,j)\setminus\tilde{\Pi}_{\geq 2}(i,i,j,j)$ must include either variables from the first and the third or the second and fourth function, which is not possible because of $i\neq j$. For $i=j$, all involved partitions $\pi\in\Pi_{\geq 2}(i,i,i,i)\setminus\tilde{\Pi}_{\geq 2}(i,i,i,i)$ cancel out with the second term, which equals $((i-1)!)^2||h_i||^4_{L^2(\lambda_t^i)}$ in this case. Because of the fixed variables $y_1$ and $y_2$, all partitions satisfy $|\pi|>\max\{i,j\}$. This finally leads to
$$R_{ij}\leq \sum_{\substack{\pi\in \tilde{\Pi}_{\geq 2}(i,i,j,j)\\ |\pi|>\max\{i,j\}}} \; \int\limits_{\Y^{|\pi|}} (h_i\otimes h_i\otimes h_j \otimes h_j)_{\pi}\;\dint\lambda_t^{|\pi|}.$$
Hence, another application of Lemma \ref{lem:intkernels} yields the bound
\begin{equation}\label{eq:bound2}
T_2=\E |\V\,U_{A_t}-\langle D U_{A_t},-D L^{-1}U_{A_t}\rangle_{L^2(\lambda_t)}|\leq A_k k!\sigma_t (\rho_t+\rho_t^3)
\end{equation}
with $A_k:=\sum_{i,j=1}^k iN_{ij}$ where $N_{ij}$ is the cardinality of $\tilde{\Pi}_{\geq 2}(i,i,j,j)$. For $T_3$ in \eqref{eq:bound0} we find 
\begin{eqnarray*}
&&\E\int\limits_{\Y} (D_yU_{A_t}(D_yU_{A_t}-1))^2\;\lambda_t(\dint y)\\ & = & \int\limits_{\Y} \E(D_yU_{A_t})^4-2\E (D_yU_{A_t})^3+\E (D_yU_{A_t})^2 \ \lambda_t(\dint y)\\
& = & \int\limits_{\Y} \E(\sum_{q=1}^k q I_{q-1}(h_q(y,\cdot)) )^4-2\E (\sum_{q=1}^k q I_{q-1}(h_q(y,\cdot)))^3\\
& & \hspace{2cm}+\E (\sum_{q=1}^k q I_{q-1}(h_q(y,\cdot)))^2 \ \lambda_t(\dint y). 
\end{eqnarray*}
From Proposition \ref{prop:productformula} and Lemma \ref{lem:intkernels}, it follows that
$$\int\limits_{\Y} \E(\sum_{q=1}^k q I_{q-1}(h_q(y,\cdot)))^\ell \; \lambda_t(\dint y)= k\sigma_t + R_\ell,\qquad\ell\in\{2,3,4\} $$ 
with $|R_\ell|\leq k^{\ell-1}A_k k!\sigma_t(\rho_t+\rho_t^{\ell-1})$, which readily implies the bound
\begin{equation}\label{eq:bound3}
T_3=\E\int\limits_{\Y} (D_yU_{A_t}(D_yU_{A_t}-1))^2 \;\lambda_t(\dint y) \leq 4k^3 A_k k!\sigma_t(\rho_t+\rho_t^3).
\end{equation}
Finally, we turn to $T_4$. Here, we have that
$$T_4=\E \int\limits_{\Y} (D_yL^{-1}U_{A_t})^2\;\lambda_t(\dint y)=\sum_{q=1}^k(q-1)! ||h_q||_{L^2(\lambda_t^q)}^2\leq \sum_{q=1}^kq! ||h_q||_{L^2(\lambda_t^q)}^2=\V\,U_{A_t},$$ which in view of \eqref{eq:varUat} leads to a bound for $T_4$. Plugging this together with the bounds \eqref{eq:bound1}--\eqref{eq:bound3} for $T_1$--$T_3$ into \eqref{eq:bound0}, leads to the desired result and completes the proof of Proposition \ref{prop:PoissonApprox}.
\end{proof}

\begin{remark}\rm
Convergence of $U_{A_t}$ to a Poisson distributed random variable as $t\rightarrow\infty$ can also be shown by the method of moments or cumulants. Using Lemma \ref{lem:intkernels}, some computation shows that the cumulants of $U_{A_t}$ converge to that of a Poisson distributed random variable; see also Theorem 4.11 in \cite{LRPeccati} for an attempt in this direction in a very special case. This technique seems to be easier than the proof above from a technical point of view. However, it gives only a weaker result, since one does not obtain a rate of convergence in this way.
\end{remark}

\begin{proof}[Proof of Theorem \ref{thm:general}.]
Taking $A_t=[0,xt^{-\gamma}]$, we see that the events $\{U_{A_t}\leq m-1\}$ and $\{t^{\gamma}F_t^{(m)}>x\}$ are equivalent. Thus, Proposition \ref{prop:PoissonApprox} can be applied and yields the bound 
\begin{eqnarray*}
&& \left|\P(t^\gamma F_t^{(m)}>x)-e^{-\beta x^{\tau}}\sum_{i=0}^{m-1}\frac{(\beta x^{\tau})^i}{i!}\right|\\ & \leq & |\beta x^{\tau}-\alpha_t(x)|+ C_k \frac{1-e^{-\alpha_t(x)}}{\alpha_t(x)}\left(1+\frac{1}{\alpha_t(x)}\right) \sqrt{\alpha_t(x)\left(r_t(x)+r_t(x)^3\right)}\notag
\end{eqnarray*} 
with a constant $C_k>0$ only depending on $k$. By \eqref{eq:alpha} and \eqref{eq:condition} there are constants $\hat{C}_{f,x},\tilde{C}_{f,x}>0$ depending on the function $f$ and on $x$ such that $$C_k \frac{1-e^{-\alpha_t(x)}}{\alpha_t(x)}  \left(1+\frac{1}{\alpha_t(x)}\right) \sqrt{\alpha_t(x)}\leq \hat{C}_{f,x}$$ and $$\sqrt{r_t(x)+r_t(x)^3}\leq\tilde{C}_{f,x}\sqrt{r_t(x)}$$ for $t\geq 1$. Putting $C_{f,x}=\hat{C}_{f,x}\cdot\tilde{C}_{f,x}$ implies part b) of the theorem. For the proof of a) we define the two set classes
 $${\cal I}=\{I=(a,b]: 0\leq a \leq b<\infty\} \; \text{ and } \; {\cal V}=\{V=\bigcup_{i=1}^n I_i:I_i\in {\cal I},\,i=1,\hdots,n,\,n\in\N\}.$$
(By convention, $(a,a]=\emptyset$.) From \cite[Thm. 16.29]{Kallenberg} we infer that convergence in distribution of the re-scaled point processes $t^\gamma \xi_t$ to the Poisson point process $\xi$ with intensity measure $\nu$ is implied by the two conditions
\begin{equation}\label{eq:conditionU}
\lim_{t\rightarrow\infty} \P(\xi_t(t^{-\gamma}V) =\emptyset)=\P(\xi(V) =\emptyset)=e^{-\nu(V)}, \qquad V\in {\cal V}
\end{equation}
and
\begin{equation}\label{eq:conditionI}
 \lim_{t\rightarrow\infty} \P(\xi_t(t^{-\gamma}I)>1)= \P(\xi(I)>1)=1-(1+\nu(I))e^{-\nu(I)}, \qquad I\in {\cal I},
\end{equation}
which are to be checked in the following. For $V\in{\cal V}$, we can assume without loss of generality that $V$ can be written as
$$V=\bigcup_{i=1}^n(a_i,b_i]\quad{\rm with}\quad 0 < a_1 < b_1 < a_2 < b_2 < \hdots < a_n < b_n.$$
We take now $A_t= t^{-\gamma}V$ in Proposition \ref{prop:PoissonApprox} and see that
\begin{eqnarray*}
\sigma_t & = & \frac{1}{k!}\int\limits_{\Y^k} \1(f(y_1,\hdots,y_k)\in t^{-\gamma}V)\;\lambda_t^k(\dint(y_1,\hdots,y_k))\\
& = & \frac{1}{k!} \sum_{i=1}^n\;\int\limits_{\Y^k} \1(f(y_1,\hdots,y_k)\in t^{-\gamma} (a_i,b_i])\;\lambda_t^k(\dint(y_1,\hdots,y_k))\\
& = & \sum_{i=1}^n (\alpha_t(b_i)-\alpha_t(a_i))\\
& \rightarrow & \sum_{i=1}^n (\beta b_i^{\tau}-\beta a_i^{\tau})= \nu(V)\qquad{\rm as}\qquad t\rightarrow\infty
\end{eqnarray*}
by condition \eqref{eq:alpha}. Moreover,
\begin{eqnarray*}
\rho_t & = & \sup_{\substack{y_1,\hdots,y_{k-j}\in\Y\\ 1\leq j\leq k-1}}\lambda_t^j(\{(\hat{y}_1,\hdots,\hat{y}_j)\in\Y^j: f(y_1,\hdots,y_{k-j},\hat{y}_1,\hdots, \hat{y}_j)\in t^{-\gamma}V\})\\
& \leq & \sup_{\substack{y_1,\hdots,y_{k-j}\in\Y\\ 1\leq j\leq k-1}}\lambda_t^j(\{(\hat{y}_1,\hdots,\hat{y}_j)\in\Y^j: f(y_1,\hdots,y_{k-j},\hat{y}_1,\hdots, \hat{y}_j)\leq t^{-\gamma} b_n\})
\end{eqnarray*}
and condition \eqref{eq:condition} implies that $\rho_t\rightarrow 0$ as $t\rightarrow\infty$. This shows that the assumptions of Proposition \ref{prop:PoissonApprox} are satisfied, whence
$$d_{\rm TV}(\xi_t(t^{-\gamma} V),Po(\nu(V)))=d_{\rm TV}(U_{t^{-\gamma}V},Po(\nu(V)))\rightarrow 0\qquad {\rm as}\qquad t\rightarrow\infty.$$
Since ${\cal I}\subset{\cal V}$, this shows \eqref{eq:conditionU} and \eqref{eq:conditionI} and completes the proof.
\end{proof}

%\subsection*{Acknowledgement}
%We would like to thank Norbert Henze (KIT) for a stimulating discussion that has drawn our attention to what we have presented above as Example 1. We also thank Matthias Reitzner (Osnabr\"uck University) for insights concerning the Poisson polytope in Example 3 and for his constant interest in our work.


\begin{thebibliography}{30}\small

\bibitem{BR10}
{\sc I. B\'ar\'any and M. Reitzner} (2010): {\em Poisson polytopes}, Ann. Probab. \textbf{38}, 1507--1531.

\bibitem{Barbour88}
{\sc A.D. Barbour} (1988): {\em Stein's method and Poisson process convergence}, J. Appl. Probab. \textbf{25A}, 175–-184.

\bibitem{BarbourBrown}
{\sc A.D. Barbour and T.C. Brown (1992)}: {\em Stein's method and point process approximation}, Stoch. Proc. Appl. \textbf{43}, 9–-31.

\bibitem{BarbourHolstJanson92}
{\sc A.D. Barbuor, L. Holst and S. Janson (1992)}: {\em Poisson approximation}, Oxford University Press, Oxford.

\bibitem{BaumstarkLast}
{\sc V. Baumstark and G. Last (2009)}: {\em Gamma distributions for stationary Poisson flat processes}, Adv. Appl. Probab. \textbf{41}, 911--939.

\bibitem{DFR11}
{\sc L. Decreusefond, E. Ferraz, H. Randriam and A. Vergne} (2011): {\em Simplicial homology of random configurations}, arXiv: 1103.4457 [math.PR].

\bibitem{GJ}
{\sc G. Grimmett and S. Janson} (2003): {\em On smallest triangles}, Rand. Struct. Alg. \textbf{23}, 206--223.

\bibitem{Henze82}
{\sc N. Henze} (1982): {\em The limit distribution for maxima of `weighted' $r$th nearest neighbour distances}, J. Appl. Probab. \textbf{19}, 344--354. 

\bibitem{Henze96}
{\sc N. Henze and T.Klein} (1996): {\em The limit distribution of the largest interpoint distance from a symmetric Kotz sample}, J. Multiv. Anal. \textbf{57}, 228--239. 

\bibitem{HugLastWeil}
{\sc D. Hug, G. Last and W. Weil (2003)}: {\em Distance measurements on processes of flats}, Adv. Appl. Probab. \textbf{35}, 70--95.

\bibitem{Janson}
{\sc S. Janson} (1987): {\em Poisson convergence and Poisson processes with applications to random graphs}, Stoch. Proc. Appl. \textbf{26}, 1--30.

\bibitem{Kallenberg}
{\sc O. Kallenberg} (2002): {\em Foundations of Modern Probability}, Second Edition, Springer, New York.

\bibitem{KMT}
{\sc S. Kanagawa, Y. Mochizuki and H. Tanaka (1992)}: {\em Limit theorems for the minimum interpoint distance between any pair of i.i.d. random points in $\R^d$}, Ann. Inst. Statist. Math. \textbf{44}, 121--131.

\bibitem{LRPeccati}
{\sc R. Lachieze-Rey and G. Peccati (2011)}: {\em Fine Gaussian fluctuations on the Poisson space, I: contractions, cumulants and geometric random graphs}, arXiv:1111.7312 [math.PR].

%\bibitem{LPST}
%{\sc G. Last, M.D. Penrose, M. Schulte and C. Th\"ale} (2012): {\em Moments and central limit theorems for some multivariate Poisson functionals}, in preparation.

\bibitem{LM07}
{\sc W. Lao and M. Mayer} (2008): {\em U-max-statistics}, J. Multivar. Anal. \textbf{9}, 2039--2052.

\bibitem{LastPenrose2011}
{\sc G. Last and M.D. Penrose (2011)}: {\em Poisson process Fock space representation, chaos expansion and covariance inequalities}, Probab. Th. Rel. Fields. \textbf{150}, 663--690.

\bibitem{MM07}
{\sc M. Mayer and I. Molchanov} (2007): {\em Limit theorems for the diameter of a random sample in the unit ball}, Extremes \textbf{10}, 151--174. 

\bibitem{Miles}
{\sc R.E. Miles} (1995): {\em A heuristic proof of a long-standing conjecture of D.G. Kendall concerning the shapes of certain large random polygons}, Adv. Appl. Prob. \textbf{27}, 397–-417.

\bibitem{Miller}
{\sc D.R. Miller} (1976): {\em Order statistics, Poisson processes and repairable systems}, J. Appl. Probab. \textbf{31}, 519--529.

\bibitem{Schmidt}
{\sc D. Neuh\"auser, C. Hirsch, C. Gloaguen and V. Schmidt} (2012): {\em On the distribution of typical shortest-path lengths in connected random geometric graphs}, to appear in Queueing Systems. 

\bibitem{NualartVives90}
{\sc D. Nualart and J. Vives (1990)}: {\em Anticipative calculus for the Poisson process based on the Fock space}, in Sem. de Proba. XXIV, LMN \textbf{1426}, 177--1993.

\bibitem{Peccati11}
{\sc G. Peccati (2011)}: {\em The Chen-Stein method for Poisson functionals}, arXiv: 1112.5051 [math.PR].

\bibitem{Peccatietal2010}
{\sc G. Peccati, J.L. Sol\'e, M.S. Taqqu and F. Utzet (2010)}: {\em Stein's method and normal approximation of Poisson functionals}, Ann. Probab. \textbf{38}, 443--478.

\bibitem{PeccatiTaqqu2010}
{\sc G. Peccati and M.S. Taqqu (2010)}: {\em Wiener Chaos: Moments, Cumulants and Diagram Formulae: A survey with computer implementation}, Springer, Berlin.

\bibitem{Penrose2003}
{\sc M.D. Penrose (2003)}: {\em Random Geometric Graphs}, Oxford University Press, Oxford.

%\bibitem{Privault09}
%{\sc N. Privault (2009)}: {\em Stochastic analysis in discrete and continuous settings with normal martingales}, Springer, Berlin.

\bibitem{R02}
{\sc M. Reitzner} (2002): {\em Random points on the boundary of smooth convex bodies}, Trans. Amer. Math. Soc. \textbf{354}, 2243--2278.

\bibitem{RS11}
{\sc M. Reitzner and M. Schulte} (2011): {\em Central limit theorems for U-statistics of Poisson point processes}, arXiv: 1104.1039 [math.PR].

\bibitem{Resnick}
{\sc S.I. Resnick} (1987): {\em Extreme values, regular variation and point processes}, Springer, New York.

\bibitem{Schneider99}
{\sc R. Schneider} (1999): {\em A duality for Poisson flats}, Adv. Appl. Probab. \textbf{31}, 63--68.

\bibitem{SW}
{\sc R. Schneider and W. Weil} (2008): {\em Stochastic and Integral Geometry}, Springer, Berlin.

\bibitem{S11}
{\sc M. Schulte} (2011): {\em A central limit theorem for the Poisson-Voronoi approximation}, arXiv: 1111.6466 [math.PR].

\bibitem{SilvermanBrown}
{\sc B. Silverman and T. Brown (1978)}: {\em Short distances, flat triangles and Poisson limits}, J. Appl. Probab. \textbf{15}, 815--825.

\bibitem{Surgailis1984}
{\sc D. Surgailis (1984)}: {\em On multiple Poisson stochastic integrals and associated Markov semigroups}, Probab. Math. Statist. \textbf{38}, 217--239.

\end{thebibliography}
\end{document}